\documentclass[a4paper,fleqn,11pt]{article}

\usepackage{amsmath}
\usepackage{amsthm}
\usepackage{amssymb}
\usepackage[a4paper,top=3cm, bottom=3cm, left=3cm, right=3cm]{geometry}
\usepackage[shortlabels]{enumitem}
\usepackage[pdftex,ocgcolorlinks,pagebackref=false]{hyperref}

\usepackage[affil-it]{authblk}

\theoremstyle{plain}
\newtheorem{theorem}{Theorem}[section]
\newtheorem{lemma}[theorem]{Lemma}
\newtheorem{proposition}[theorem]{Proposition}
\newtheorem{corollary}[theorem]{Corollary}

\theoremstyle{definition}
\newtheorem{definition}[theorem]{Definition}

\newcommand{\setbuild}[2]{\left\{#1\middle|#2\right\}}

\newcommand{\preorderle}{\preccurlyeq}
\newcommand{\preorderge}{\succcurlyeq}
\newcommand{\asymptoticle}{\precsim}
\newcommand{\asymptoticge}{\succsim}
\newcommand{\aasymptoticle}{\precapprox}
\newcommand{\aasymptoticge}{\succapprox}

\DeclareMathOperator{\ev}{ev}

\title{A generalization of Strassen's spectral theorem}
\author[1,2]{P\'eter Vrana}
\affil[1]{Institute of Mathematics, Budapest University of Technology and Economics, Egry~J\'ozsef u.~1., Budapest, 1111 Hungary.}
\affil[2]{MTA-BME Lend\"ulet Quantum Information Theory Research Group}

\begin{document}
\maketitle
\begin{abstract}
Given a semiring with a preorder subject to certain conditions, the asymptotic spectrum, as introduced by Strassen (J. reine angew. Math. 1988), is a compact Hausdorff space together with a map from the semiring to the ring of continuous functions, which contains all information required to asymptotically compare large powers of the elements.

Compactness of the asymptotic spectrum is closely tied with a boundedness condition assumed in Strassen's work. In this paper we present a generalization that relaxes this condition while still allowing asymptotic comparison via continuous functions on a locally compact Hausdorff space.
\end{abstract}

\section{Introduction}

Motivated by the study of asymptotic tensor rank and more generally relative bilinear complexity, in \cite{strassen1988asymptotic} Strassen developed the theory of asymptotic spectra of preordered semirings (see also \cite[Chapter 2]{zuiddam2018algebraic} for a recent exposition). The spectrum of a (commutative, unital) semiring $S$ with respect to a preorder $\preorderle$ is the set $\Delta(S,\preorderle)$ of $\preorderle$-monotone homomorphisms $S\to\mathbb{R}_{\ge 0}$, i.e. maps $f$ satisfying $f(x+y)=f(x)+f(y)$, $f(xy)=f(x)f(y)$, $f(1)=1$ and $x\preorderle y\implies f(x)\le f(y)$. Strassen introduces the asymptotic preorder $\asymptoticle$ as $x\asymptoticge y$ iff there is a sublinear nonnegative integer sequence $(k_n)_{n\in\mathbb{N}}$ such that for all $n$ the inequality $2^{k_n}x^n\preorderge y^n$ holds (see Section~\ref{sec:asymptoticpreorder} for precise definitions). Every element $f\in\Delta(S,\preorderle)$ satisfies $x\asymptoticge y\implies f(x)\ge f(y)$. The key insight is that under an additional boundedness assumption the converse also holds in the sense that $\left(\forall f\in\Delta(S,\preorderle):f(x)\ge f(y)\right)\implies x\asymptoticge y$ (see the precise statement below).

In the main application in \cite{strassen1988asymptotic} the role of $S$ is played by the set of (equivalence classes of) tensors over a fixed field and of fixed order but arbitrary finite dimension. The operations are the direct sum and the tensor (Kronecker) product of tensors and the preorder is given by tensor restriction (alternatively: degeneration). Asymptotic tensor rank can be characterized in terms of the resulting asymptotic restriction preorder, and as a consequence also in terms of the asymptotic spectrum.

More recently, the theory of asymptotic spectra has been applied to a range of other problems as well. In \cite{zuiddam2019asymptotic} Zuiddam introduced the asymptotic spectrum of graphs and found a dual characterization of the Shannon capacity. Continuing this line of research, Li and Zuiddam \cite{li2018quantum} studied the quantum Shannon capacity and entanglement-assisted quantum capacity of graphs as well as the entanglement-assisted quantum capacity of noncommutative graphs via the asymptotic spectra of suitable semirings of graphs and noncommutative graphs. In the context of quantum information theory, tensors in Hilbert spaces represent entangled pure states, and the relevant preorder is given by local operations and classical communication. This again gives rise to a preordered semiring \cite{jensen2019asymptotic} which refines the tensor restriction preorder and which also fits into Strassen's framework. In this application, the asymptotic spectrum provides a characterization of converse error exponents for asymptotic entantanglement transformations.

At the same time, it became clear that the boundedness condition in Strassen's theorem is too restrictive for some purposes. To motivate the need for a relaxation of this condition, note first that the concept of asymptotic preorder does not make use of the additive structure and in the applications one is mainly interested in the ordered commutative monoid $S\setminus\{0\}$ with multiplication and the preorder. Such objects have been studied in \cite{fritz2017resource} as a mathematical model for resource theories. In this context, asymptotic properties such as the asymptotic preorder above and more generally, rate formulas are of central interest. \cite[Theorem 8.24.]{fritz2017resource} implies that if an ordered commutative monoid has an element $u$ such that for every $x$ there is $k$ such that $u^k\preorderge x$ and $u^kx\preorderge 1$, then $u^{o(n)}x^n\preorderge y^n$ for all large $n$ iff for all monotone homomorphisms $f$ into the semigroup $\mathbb{R}_{\ge0}$ (with multiplication and the usual order) the inequality $f(x)\ge f(y)$ holds. At this point the relation to Strassen's result should be clear: if the monoid in question is the multiplicative monoid of a preordered semiring and the above condition is satisfied with $u=2$, then one can restrict to monotones that preserve both operations. One may wonder if $u=2$ is necessary for this stronger conclusion to hold, but simple examples show that without this assumption the set of monotone semiring homomorphisms can be too small to characterize the asymptotic preorder. In this spirit, Fritz proved a generalization of Strassen's theorem \cite{fritz2018generalization}, listing conditions that are equivalent to $\forall f\in\Delta(S,\preorderle):f(x)\ge f(y)$ and generalize the asymptotic preorder in different ways, emphasizing also the similarity to Positivstellensätze.

Our main result is a sufficient condition under which the spectrum does characterize the asymptotic preorder as above and which generalizes Strassen's condition. Before stating our main theorem, let us recall Strassen's result in a form that eases comparison. We say that an element $u\in S$ is power universal \cite[Definition 2.7.]{fritz2018generalization} if for every $x\in S\setminus\{0\}$ there is a $k\in\mathbb{N}$ such that $u^kx\preorderge 1$ and $u^k\preorderge x$. If $S$ contains a power universal element than $S$ is said to be of polynomial growth. The asymptotic preorder is defined as $x\asymptoticge y$ iff there is a sublinear nonnegative integer sequence $(k_n)_{n\in\mathbb{N}}$ such that for all $n$ the inequality $u^{k_n}x^n\preorderge y^n$ holds. Note that the asymptotic preorder does not depend on the choice of the power universal element (see Lemma~\ref{lem:changepoweruniversal} below). With these definitions Strassen's theorem on asymptotic spectra can be stated as follows.
\begin{theorem}[Strassen {\cite[Corollary 2.6.]{strassen1988asymptotic}}, Zuiddam {\cite[Theorem 2.12]{zuiddam2018algebraic}}]\label{thm:Strassen}
Let $(S,\preorderle)$ be a preordered semiring such that the canonical map $\mathbb{N}\hookrightarrow S$ is an order embedding, and suppose that $u=2$ is power universal. Then for every $x,y\in S$ we have
\begin{equation}
x\asymptoticge y\iff\forall f\in\Delta(S,\preorderle):f(x)\ge f(y).
\end{equation}

$\Delta(S,\preorderle)$ is a nonempty compact Hausdorff space.
\end{theorem}

Our main result the following (for $s\in S$, $\ev_s:\Delta(S,\preorderle)$ denotes the evaluation map $f\mapsto f(s)$).
\begin{theorem}\label{thm:main}
Let $(S,\preorderle)$ be a preordered semiring of polynomial growth such that the canonical map $\mathbb{N}\hookrightarrow S$ is an order embedding, and let $M\subseteq S$ and $S_0$ the subsemiring generated by $M$. Suppose that
\begin{enumerate}[({M}1)]
\item\label{it:invertibleuptobounded} for all $s\in S\setminus\{0\}$ there exist $m\in M$ and $n\in\mathbb{N}$ such that $1\preorderle nms$ and $ms\preorderle n$
\item\label{it:boundedev} for all $m\in M$ such that $\ev_m:\Delta(S_0,\preorderle)\to\mathbb{R}_{\ge 0}$ is bounded there is an $n\in\mathbb{N}$ such that $m\preorderle n$.
\end{enumerate}
Then for every $x,y\in S$ we have
\begin{equation}
x\asymptoticge y\iff\forall f\in\Delta(S,\preorderle):f(x)\ge f(y).
\end{equation}

$\Delta(S,\preorderle)$ is a nonempty locally compact Hausdorff space and if $u$ is power universal then $\ev_u:\Delta(S,\preorderle)\to\mathbb{R}_{\ge 0}$ is proper.
\end{theorem}
This result is a generalization of Theorem~\ref{thm:Strassen}: if $u=2$ is power universal then one can choose $M=\{1\}$ for which both conditions are easily verified, whereas the topological part follows from the fact that $\ev_u$ is a constant map in this case.

It may happen that \ref{it:invertibleuptobounded} is not satisfied by any subset $M$ of $S$ but can be satisfied after localizing at a suitable multiplicative set $T$. We will see that localization does not affect the asymptotic preorder, which leads to the following corollary, a somewhat more flexible version of our main result:
\begin{corollary}\label{cor:main}
Let $(S,\preorderle)$ be a preordered semiring of polynomial growth such that $\mathbb{N}\hookrightarrow S$ is an order embedding. Let $M\subseteq S$ and $T\subseteq S\setminus\{0\}$ be a multiplicative set containing $1$, and let $S_0$ be the subsemiring generated by $M\cup T$. Suppose that
\begin{enumerate}[({M}1')]
\item\label{it:fracinvertibleuptobounded} for all $s\in S\setminus\{0\}$ there exist $m\in M$, $t_1,t_2\in T$ and $n\in\mathbb{N}$ such that $t_2\preorderle n mt_1s$ and $mt_1s\preorderle n t_2$
\item\label{it:fracboundedev} for all $m\in M$ and $t_1,t_2\in T$ such that $\ev_m\frac{\ev_{t_1}}{\ev_{t_2}}:\Delta(S_0,\preorderle)\to\mathbb{R}_{\ge 0}$ is bounded there is an $n\in\mathbb{N}$ such that $mt_1\preorderle n t_2$. Then for every $x,y\in S$ we have
\begin{equation}\label{eq:maincor}
x\asymptoticge y\iff\forall f\in\Delta(S,\preorderle):f(x)\ge f(y).
\end{equation}
\end{enumerate}
\end{corollary}
In particular, choosing $M=\{1\}$ and $T=S\setminus\{0\}$ (effectively repacing $S$ with its semifield of fractions) guarantees \ref{it:fracinvertibleuptobounded}, but then verifying \ref{it:fracboundedev} (which is also a necessary condition in this case) may require a fairly detailed knowledge of $\Delta(S,\preorderle)$. In practice it is desirable to choose $M$ and $T$ in such a way that $S_0$ is as simple as possible. Note however that having a complete classification of the elements of $\Delta(S_0,\preorderle)$ is not a prerequisite for verifying \ref{it:fracboundedev}.

Under the conditions of Theorem~\ref{thm:Strassen}, the asymptotic spectrum has a certain uniqueness (or minimality) property \cite[Corollary 2.7.]{strassen1988asymptotic}. To state the version that is true in the more general setting of Theorem~\ref{thm:main}, note first that the map $s\mapsto\ev_s$ is a semiring homomorphism into $C(\Delta(S,\preorderle))$, the space of continuous functions on $\Delta(S,\preorderle)$, monotone with respect to the pointwise partial order. Given a preordered semiring $(S,\preorderle)$ of polynomial growth, let us call a pair $(X,\Phi)$ an abstract asymptotic spectrum for $(S,\preorderle)$ if $X$ is a locally compact topological space, $\Phi:S\to C(X)$ a semiring homomorphism such that $\Phi(S)$ separates the points of $X$, the image of every power universal element is a proper map, and $\forall x,y\in S:x\asymptoticge y\iff\Phi(x)\ge\Phi(y)$.
\begin{proposition}\label{prop:uniqueness}
Let $(X,\Phi)$ be an abstract asymptotic spectrum for $(S,\preorderle)$. Then there is a unique homeomorphism $h:X\to\Delta(S,\preorderle)$ such that $\forall s\in S:\Phi(s)=\ev_s\circ h$.
\end{proposition}

The paper is structured as follows. Section~\ref{sec:asymptoticpreorder} studies basic properties of preordered semirings, including the asymptotic preorder and its relation to localization. In Section~\ref{sec:spectrum} we define the spectrum as a topological space and study the continuous maps between spectra induced by monotone homomorphisms between preordered semirings. Section~\ref{sec:mainproof} contains the proof of our main result, Theorem~\ref{thm:main}, and of Proposition~\ref{prop:uniqueness}.

\section{Asymptotic preorder}\label{sec:asymptoticpreorder}

By a semiring we mean a set equipped with binary operations $+$ and $\cdot$ that are commutative and associative and have neutral elements $0$ and $1$ such that $0\cdot a=a$ and $a(b+c)=ab+ac$ for any elements $a,b,c$. A semiring homomorphism $\varphi:S_1\to S_2$ is a map satisfying $\varphi(0)=0$, $\varphi(1)=1$, $\varphi(a+b)=\varphi(a)+\varphi(b)$ and $\varphi(ab)=\varphi(a)\varphi(b)$ for $a,b\in S_1$.
\begin{definition}
A preordered semiring is a pair $(S,\preorderle)$ where $S$ is a semiring, $\preorderle$ is a transitive and reflexive relation on $S$ such that $0\preorderle 1$ and when $a,b,c\in S$ satisfy $a\preorderle b$ then $a+c\preorderle b+c$ and $ac\preorderle bc$.

Let $(S_1,\preorderle_1)$ and $(S_2,\preorderle_2)$ be preordered semirings. A semiring homomorphism $\varphi:S_1\to S_2$ is monotone if $a,b\in S_1$, $a\preorderle_1 b$ implies $\varphi(a)\preorderle_2\varphi(b)$.
\end{definition}

\begin{definition}[{\cite[Definition 2.7.]{fritz2018generalization}}]
Let $(S,\preorderle)$ be a preordered semiring. An element $u\in S$ is power universal if $u\preorderge 1$ and for every $x\in S$ there is a $k\in\mathbb{N}$ such that $x\preorderle u^k$ and $u^kx\preorderge 1$.

If such an element exists then $S$ is said to be of polynomial growth.
\end{definition}
It is clear that any element larger than a power universal one is also power universal. More generally, if $u'\preorderge 1$ and $u\preorderle (u')^k$ for some $k\in\mathbb{N}$ then $u'$ is also power universal.

With the help of a power universal element we can define a generalization of the asymptotic preorder \cite[(2.7)]{strassen1988asymptotic} as follows.
\begin{definition}\label{def:asymptoticge}
Let $(S,\preorderle)$ be a preordered semiring and $u\in S$ a power universal element. The asymptotic preorder $\asymptoticle_u$ is defined as $x\asymptoticge_u y$ iff there is a sequence $(k_n)_{n\in\mathbb{N}}$ such that
\begin{equation}
\lim_{n\to\infty}\frac{k_n}{n}=0
\end{equation}
and
\begin{equation}\label{eq:asymptoticgedef}
\forall n\in\mathbb{N}:u^{k_n}x^n\preorderge y^n.
\end{equation}
\end{definition}
Since $u\preorderge 1$, we may assume whenever convenient that $(k_n)_{n\in\mathbb{N}}$ is nondecreasing. By multiplying the inequalities we may replace $k_n$ with a subadditive sequence, or even require the inequality in \eqref{eq:asymptoticgedef} only for infinitely many $n$.

The asymptotic preorder is defined in terms of a power universal element. However, changing to a different power universal element does not affect the resulting preorder, as the following lemma shows.
\begin{lemma}\label{lem:changepoweruniversal}
If $u_1$ and $u_2$ are power universal elements of $S$ then $\asymptoticle_{u_1}=\asymptoticle_{u_2}$.
\end{lemma}
\begin{proof}
By symmetry it is enough to prove $\asymptoticle_{u_1}\supseteq\asymptoticle_{u_2}$. Let $k\in\mathbb{N}$ such that $u_2\preorderle u_1^k$. Suppose that $x\asymptoticge_{u_2}y$. This means there is a sublinear sequence $(k_n)_{n\in\mathbb{N}}$ of natural numbers such that \eqref{eq:asymptoticgedef} holds with $u_2$. Then we have $u_1^{kk_n}x^n\preorderge u_2^{k_n}x^n\preorderge y^n$ for all $n$ and $kk_n/n\to 0$, therefore $x\asymptoticge_{u_1}y$.
\end{proof}

By Lemma~\ref{lem:changepoweruniversal} the asymptotic preorder is determined by the preordered semiring of polynomial growth even without specifying a power universal element. For this reason we will drop the subscript from the notation and write $\asymptoticle$ for the asymptotic preorder.

The following pair of lemmas show basic properties of the asymptotic preorder. Some of these are analogous to those in \cite[Lemma 2.3., Lemma 2.4.]{zuiddam2018algebraic}, with nearly identical proofs.
\begin{lemma}\label{lem:asymptoticproperties}
Let $(S,\preorderle)$ be a preordered semiring and $u$ a power universal element. Then
\begin{enumerate}[(i)]
\item\label{it:asymptoticlarger} $\preorderle\subseteq\asymptoticle$.
\item\label{it:asymptoticpreorderedsemiring} $(S,\asymptoticle)$ is a preordered semiring.
\item\label{it:asymptoticpolygrowth} $u$ is power universal with respect to the asymptotic preorder.
\item\label{it:asymptoticcancellation} Suppose that $x,y\in S$ and there is an $s\in S\setminus\{0\}$ such that $sx\preorderge sy$. Then $x\asymptoticge y$.
\item\label{it:asymptoticsmallfactors} Let $x,y,s,t\in S\setminus\{0\}$ and suppose that for all $n\in\mathbb{N}$ the inequality $sx^n\preorderge ty^n$ holds. Then $x\asymptoticge y$.
\end{enumerate}
\end{lemma}
\begin{proof}
\ref{it:asymptoticlarger}: If $x\preorderge y$ then \eqref{eq:asymptoticgedef} is satisfied with $k_n=0$.
\ref{it:asymptoticpreorderedsemiring}: Reflexivity and $0\asymptoticle 1$ follows from \ref{it:asymptoticlarger}. To prove transitivity let $x\asymptoticge y$ and $y\asymptoticge z$. Choose sublinear sequences $(k_n)_{n\in\mathbb{N}}$ and $(l_n)_{n\in\mathbb{N}}$ such that $u^{k_n}x^n\preorderge y^n$ and $u^{l_n}y^n\preorderge z^n$ for all $n$. Then $n\mapsto k_n+l_n$ is also sublinear and $u^{k_n+l_n}x^n\preorderge u^{l_n}y^n\preorderge z^n$, therefore $x\asymptoticge z$. We prove compatibility with the operations. Let $x\asymptoticge y$ and $z\in S$, and choose $(k_n)_{n\in\mathbb{N}}$ sublinear and nondecreasing such that $u^{k_n}x^n\preorderge y^n$ for all $n$. Then also $u^{k_n}(xz)^n\preorderge (yz)^n$ for all $n$, therefore $xz\asymptoticge yz$. For the sum we use
\begin{equation}
\begin{split}
u^{k_n}(x+z)^n
 & = u^{k_n}\sum_{m=0}^n\binom{n}{m}x^mz^{n-m}  \\
 & \preorderge \sum_{m=0}^n\binom{n}{m}u^{k_m}x^mz^{n-m}  \\
 & \preorderge \sum_{m=0}^n\binom{n}{m}y^mz^{n-m}  \\
 & = (y+z)^n,
\end{split}
\end{equation}
therefore $x+z\asymptoticge y+z$.
\ref{it:asymptoticpolygrowth} is an immediate consequence of \ref{it:asymptoticlarger}.
\ref{it:asymptoticcancellation}: Let $k\in\mathbb{N}$ be such that $u^ks\preorderge 1$ and $u^k\preorderge s$. By induction, $sx^n\preorderge sx^{n-1}y\preorderge\ldots\preorderge sy^n$, and therefore $u^kx^n\preorderge y^n$ for all $n$. This implies $x\asymptoticge y$.
\ref{it:asymptoticsmallfactors}: Let $k,l\in\mathbb{N}$ be such that $u^kt\preorderge 1$ and $u^l\preorderge s$. Then $u^{k+l}x^n\preorderge u^ksx^n\preorderge u^kty^n\preorderge y^n$ for all $n$, therefore \eqref{eq:asymptoticgedef} is satisfied with $k_n=k+l$.
\end{proof}

By \ref{it:asymptoticpreorderedsemiring} and \ref{it:asymptoticpolygrowth} of Lemma~\ref{lem:asymptoticproperties} we can iterate the construction and form the asymptotic preorder $\aasymptoticle$ that compares large powers via $\asymptoticle$. However, the following lemma tells us that this does not give anything new.
\begin{lemma}\label{lem:aasymptotic}
$\aasymptoticle=\asymptoticle$.
\end{lemma}
\begin{proof}
We have $\asymptoticle\subseteq\aasymptoticle$ by \ref{it:asymptoticlarger} of Lemma~\ref{lem:asymptoticproperties}. Let $x\aasymptoticge y$ and choose $(k_n)_{n\in\mathbb{N}}$ sublinear such that $u^{k_n}x^n\asymptoticge y^n$ for all $n$. Let $(l_{n,m})_{n,m\in\mathbb{N}}$ be such that for all $n$ $\lim_{m\to\infty}l_{n,m}/m=0$ and for all $n,m$
\begin{equation}
u^{l_{n,m}+mk_n}x^{nm}=u^{l_{n,m}}(u^{k_n}x^n)^m\preorderge (y^n)^m=y^{nm}
\end{equation}
holds. Choose a sequence $n\mapsto m_n$ such that $l_{n,m_n}\le m_n$ (e.g. let $m_n$ be the first index such that $l_{n,m_n}\le m_n$). Then
\begin{equation}
\lim_{n\to\infty}\frac{l_{n,m_n}+m_nk_n}{nm_n}\le\lim_{n\to\infty}\frac{1}{n}+\frac{k_n}{n}=0,
\end{equation}
therefore $x\asymptoticge y$.
\end{proof}

Our next goal is to understand how the asymptotic preorder interacts with localization. Let $T\subseteq S\setminus\{0\}$ be a multiplicative subset (i.e. $t_1,t_2\in T\implies t_1t_2\in T$) containing $1$. The localization of $S$ at $T$, denoted $T^{-1}S$ is the semiring $S\times T$ modulo the equivalence relation $(s_1,t_1)\sim(s_2,t_2)$ iff there exists $r\in T$ such that $rs_1t_2=rs_2t_1$, equipped with the operations $(s_1,t_1)+(s_2,t_2)=(s_1t_2+s_2t_1,t_1t_2)$, $(s_1,t_1)\cdot(s_2,t_2)=(s_1s_2,t_1t_2)$. We denote the equivalence class of $(s,t)$ by $\frac{s}{t}$. There is a canonical homomorphism $S\to T^{-1}S$ that sends $s$ to $\frac{s}{1}$. The additive and multiplicative neutral elements are $\frac{0}{1}$ and $\frac{1}{1}$.
\begin{lemma}\label{lem:localizationpreorderwelldefined}
Let $s_1,s_2,s_1',s_2'\in S$ and $t_1,t_2,t_1',t_2'\in T$ such that $\frac{s_1}{t_1}=\frac{s_1'}{t_1'}$ and $\frac{s_2}{t_2}=\frac{s_2'}{t_2'}$. Then $\exists r\in T:rs_1t_2\preorderge rs_2t_1$ iff $\exists r'\in T:r's_1't_2'\preorderge r's_2't_1'$.
\end{lemma}
\begin{proof}
The roles of the primed and unprimed elements are symmetric, therefore it is enough to show that $rs_1t_2\preorderge rs_2t_1$ with $r\in T$ implies $\exists r'\in T:r's_1't_2'\preorderge r's_2't_1'$. Let $q_1,q_2\in T$ such that $q_1s_1t_1'=q_1s_1't_1$ and $q_2s_2t_2'=q_2s_2't_2$. Let $r'=(q_1q_2s_1t_2r)$. Then
\begin{equation}
\begin{split}
r's_1't_2'
 & = q_1q_2s_1t_2rs_1't_2'  \\
 & = q_1q_2s_1't_2'(rs_1t_2)  \\
 & \preorderge q_1q_2s_1't_2'(rs_2t_1)  \\
 & = (q_1t_1s_1')(q_2s_2t_2')r  \\
 & = (q_1t_1's_1)(q_2s_2't_2)r  \\
 & = (q_1q_2s_1t_2)rs_2't_1'  \\
 & = r's_2't_1'.
\end{split}
\end{equation}
\end{proof}

According to Lemma~\ref{lem:localizationpreorderwelldefined} we can define a relation $\preorderle$ on $T^{-1}S$ as
\begin{equation}\label{eq:localizationpreorderdef}
\frac{s_1}{t_1}\preorderge\frac{s_2}{t_2}\iff\exists r\in T:rs_1t_2\preorderge rs_2t_1.
\end{equation}

\begin{lemma}\leavevmode
\begin{enumerate}[(i)]
\item\label{it:localizationpreorderedsemiring} $(T^{-1}S,\preorderle)$ is a preordered semiring.
\item\label{it:localizationmapmonotone} The canonical homomorphism $s\mapsto\frac{s}{1}$ is monotone.
\item\label{it:localizationpoweruniversal} If $u\in S$ is power universal then $\frac{u}{1}$ is power universal in $T^{-1}S$.
\end{enumerate}
\end{lemma}
\begin{proof}
\ref{it:localizationpreorderedsemiring}: For $\frac{s}{t}\in T^{-1}S$ we have $st\preorderge st$, therefore $\frac{s}{t}\preorderge\frac{s}{t}$, i.e. $\preorderle$ is reflexive. Suppose that $\frac{s_1}{t_1}\preorderge\frac{s_2}{t_2}$ and $\frac{s_2}{t_2}\preorderge\frac{s_3}{t_3}$. This means $rs_1t_2\preorderge rs_2t_1$ and $qs_2t_3\preorderge qs_3t_2$ for some $r,q\in T$. Therefore
\begin{equation}
(qrt_2)s_1t_3=q(rs_1t_2)t_3 \preorderge q(rs_2t_1)t_3=rt_1(qs_2t_3)= \preorderge qrs_3t_1t_2=(qrt_2)s_3t_1,
\end{equation}
which implies $\frac{s_1}{t_1}\preorderge\frac{s_3}{t_3}$, i.e. $\preorderle$ is transtive.
We have $\frac{0}{1}\preorderle\frac{1}{1}$ since $0\preorderle 1$.
We prove compatibility with the operations. Suppose that $\frac{s_1}{t_1}\preorderge\frac{s_2}{t_2}$ and let $\frac{s'}{t'}\in T^{-1}S$. This means $rs_1t_2\preorderge rs_2t_1$ for some $r\in T$. Then $r(s_1t'+s't_1)t_2\preorderge r(s_2t'+s't_2)t_1$, therefore
\begin{equation}
\frac{s_1}{t_1}+\frac{s'}{t'}=\frac{s_1t'+s't_1}{t_1t'}\preorderge\frac{s_2t'+s't_2}{t_2t'}=\frac{s_2}{t_2}+\frac{s'}{t'},
\end{equation}
and $rs_1t_2s't'\preorderge rs_2t_1s't'$, therefore
\begin{equation}
\frac{s_1}{t_1}\frac{s'}{t'}=\frac{s_1s'}{t_1t'}\preorderge\frac{s_2s'}{t_2t'}=\frac{s_2}{t_2}\frac{s'}{t'}.
\end{equation}
\ref{it:localizationmapmonotone}: If $x,y\in S$ and $x\preorderge y$ then $\frac{x}{1}\preorderge\frac{y}{1}$ since \eqref{eq:localizationpreorderdef} is satisfied for $r=1$.
\label{it:localizationpoweruniversal}: Let $u\in S$ be power universal, $x=\frac{s}{t}\in T^{-1}S\setminus\{\frac{0}{1}\}$. Choose $k_s,k_t$ such that the inequalities $s\preorderle u^{k_s}$, $u^{k_s}s\preorderge 1$, $t\preorderle u^{k_t}$, $u^{k_t}t\preorderge 1$ hold. Then $s\preorderle u^{k_s}\preorderle u^{k_s+k_t}t$, therefore
\begin{equation}
\frac{s}{t}\preorderle\left(\frac{u}{1}\right)^{k_s+k_t},
\end{equation}
and $t\preorderle u^{k_t}\preorderle u^{k_s+k_t}s$, therefore
\begin{equation}
\frac{1}{1}\preorderle\left(\frac{u}{1}\right)^{k_s+k_t}\frac{s}{t}.
\end{equation}
\end{proof}
It should be noted that even though we use the same symbol for the preorder on $S$ and that induced on $T^{-1}S$, it is in general not true that $x\preorderge y$ iff $\frac{x}{1}\preorderge\frac{y}{1}$. When $S$ is of polynomial growth, we denote by $\asymptoticge$ the asymptotic preorder on $T^{-1}S$. Note that this could mean two different relations, depending on whether we form the asymptotic preorder on the localization or consider the preorder induced on $T^{-1}S$ by the asymptotic preorder on $S$. However, these two relations turn out to be the same. If $rs_1t_2\asymptoticge rs_2t_1$ then $r^nu^{o(n)}s_1^nt_2^n\preorderge r^ns_2^nt_1^n$ for all $n$, i.e.
\begin{equation}\label{eq:localizationasymptoticpreorder}
\left(\frac{u}{1}\right)^{o(n)}\left(\frac{s_1}{t_1}\right)^{n}\preorderge\left(\frac{s_2}{t_2}\right)^{n}
\end{equation}
Conversely, if \eqref{eq:localizationasymptoticpreorder} holds, then there is a sequence $(r_n)_{n\in\mathbb{N}}$ in $T$ such that $r_nu^{o(n)}(s_1t_2)^n\preorderge r_n(s_2t_1)^n$, which implies $s_1t_2\asymptoticge s_2t_3$ by \ref{it:asymptoticcancellation} of Lemma~\ref{lem:asymptoticproperties} and Lemma~\ref{lem:aasymptotic}. A similar argument in the following lemma shows that the asymptotic preorder is essentially the same on $S$ and on its image in $T^{-1}S$.
\begin{lemma}\label{lem:localizationasymptotic}
For $x,y\in S$ we have $x\asymptoticge y$ iff $\frac{x}{1}\asymptoticge \frac{y}{1}$.
\end{lemma}
\begin{proof}
Suppose that $x\asymptoticge y$, i.e. $u^{k_n}x^n\preorderge y^n$ for some sublinear sequence $k_n$ and all $n$. This implies $\left(\frac{u}{1}\right)^{k_n}\left(\frac{x}{1}\right)^{n}\preorderge\left(\frac{y}{1}\right)^{n}$, therefore $\frac{x}{1}\asymptoticge \frac{y}{1}$.

Suppose that $\frac{x}{1}\asymptoticge \frac{y}{1}$. This means there is a sublinear sequence $k_n$ and elements $r_n\in S\setminus\{0\}$ such that
\begin{equation}
r_nu^{k_n}x^n\preorderge r_ny^n
\end{equation}
for all $n\in\mathbb{N}$. By \ref{it:asymptoticcancellation} of Lemma~\ref{lem:asymptoticproperties} and Lemma~\ref{lem:aasymptotic} this implies $x\asymptoticge y$.
\end{proof}
According to Lemma~\ref{lem:localizationasymptotic}, we may safely identify $x$ with $\frac{x}{1}$ for the purposes of studying the asymptotic preorder, even though the canonical homomorphism $S\to T^{-1}S$ is in general not injective.

We conclude this section with a construction that enlarges a preorder by forcing an ordering on certain pairs of elements.
\begin{definition}\label{def:Rpreorderdef}
Let $(S,\preorderle)$ be a preordered semiring and $R\subseteq S\times S$ a relation. We define the relation $\preorderle_R$ on $S$ as $a\preorderle_R b$ iff there is an $n\in\mathbb{N}$ and $s_1,\ldots,s_n,x_1,\ldots,x_n,y_1,\ldots,y_n\in S$ such that
\begin{equation}
\forall i\in[n]:(x_i,y_i)\in R
\end{equation}
and
\begin{equation}\label{eq:Rpreorderdef}
a+s_1y_1+\cdots+s_ny_n\preorderle b+s_1x_1+\cdots+s_nx_n.
\end{equation}
\end{definition}

\begin{lemma}\label{lem:Rpreorderproperties}
\leavevmode
\begin{enumerate}[(i)]
\item\label{it:Rpreorderlarger} $\preorderle\subseteq\preorderle_R$
\item\label{it:RpreordercontainsR} $R\subseteq\preorderle_R$
\item\label{it:Rpreordersemiring} $(S,\preorderle_R)$ is a preordered semiring.
\item\label{it:Rpreorderunion} If $R_1,R_2\subseteq S\times S$ then $\preorderle_{R_1\cup R_2}=(\preorderle_{R_1})_{R_2}$.
\item\label{it:Rpreorderfinite} $\displaystyle\preorderle_R=\bigcup_{\substack{R'\subseteq R  \\  |R'|<\infty}}\preorderle_{R'}$.
\end{enumerate}
\end{lemma}
\begin{proof}
\ref{it:Rpreorderlarger}: If $a\preorderle b$ then we may choose $n=0$ in the definition, therefore $a\preorderle_R b$.
\ref{it:RpreordercontainsR}: If $(x,y)\in R$ then we can choose $n=1$, $s_1=1$, $x_1=x$, $y_1=y$ so that \eqref{eq:Rpreorderdef} becomes $x+y\preorderle y+x$, therefore $x\preorderle_R y$.
\ref{it:Rpreordersemiring}: By \ref{it:Rpreorderlarger} the relation $\preorderle_R$ is reflexive. We prove transitivity. Let $a\preorderle_R b$ and $b\preorderle_R c$. Then there are $n,m\in\mathbb{N}$ and elements $x_1,\ldots,x_n,y_1,\ldots,y_n,z_1,\ldots,z_m,w_1,\ldots,w_m\in S_0$ and $s_1,\ldots,s_n,t_1,\ldots,t_m\in S$ such that $(x_i,y_i)\in R$ and $(z_i,w_i)\in R$ for all $i$ in $[n]$, respectively $[m]$, and
\begin{align}
a+s_1y_1+\cdots+s_ny_n & \preorderle b+s_1x_1+\cdots+s_nx_n  \\
b+t_1w_1+\cdots+t_nw_m & \preorderle c+t_1z_1+\cdots+t_nz_m.
\end{align}
After adding $t_1w_1+\cdots+t_nw_m$ to both sides of the first inequality and $s_1x_1+\cdots+s_nx_n$ to both sides of the second one we can chain the inequalities and conclude $a\preorderle_R c$.

The compatibility with addition and multiplication can be seen directly by adding $c$ to (respectively multiplying by $c$) both sides of \eqref{eq:Rpreorderdef}.
\ref{it:Rpreorderunion}: Expanding the definition, we see that $a(\preorderle_{R_1})_{R_2}b$ means that there are natural numbers $n,n'$ and elements $s_1,\ldots,s_n,x_1,\ldots,x_n,y_1,\ldots,y_n S$ and $s'_1,\ldots,s'_{n'},x'_1,\ldots,x'_{n'},y'_1,\ldots,y'_{n'}\in S$ such that for all $i$ $(x_i,y_i)\in R_1$ and $(x'_i,y'_i)\in R_2$ and
\begin{equation}
(a+s_1y_1+\cdots+s_ny_n)+s'_1y'_1+\cdots+s'_{n'}y'_{n'}\preorderle (b+s_1x_1+\cdots+s_nx_n)+s'_1x'_1+\cdots+s'_{n'}x'_{n'},
\end{equation}
which is clearly equivalent to $a\preorderle_{R_1\cup R_2}$.
\ref{it:Rpreorderfinite} follows from the fact that \eqref{eq:Rpreorderdef} involves only finitely many pairs $(x_i,y_i)\in R$.
\end{proof}

\section{Spectrum}\label{sec:spectrum}

In this section we define and study the spectrum of a preordered semiring, i.e. the set of monotone semiring homomorphisms with the topology generated by the evaluation maps. We begin with properties related to compactness, and show that semirings of polynomial growth have locally compact spectra. Then we show that neither replacing the preorder with its asymptotic preorder nor localization affects the spectrum. Finally we prove that the inclusion of a certain type of subsemiring induces a surjective map on the spectra.

\begin{definition}
Let $(S,\preorderle)$ be a preordered semiring. The spectrum $\Delta(S,\preorderle)$ is the set of monotone semiring homomorphisms $S\to\mathbb{R}_{\ge 0}$ equipped with the initial topology with respect to the family of evaluation maps $\ev_s:\Delta(S,\preorderle)\to\mathbb{R}_{\ge 0}$ $(s\in S)$. For a subset $X\subseteq S$ we define $\ev_X:\Delta(S,\preorderle)\to\mathbb{R}_{\ge 0}^X$ as $f\mapsto(f(x))_{x\in X}$. Elements of the spectrum will be referred to as spectral points.

Let $(S_1,\preorderle_1)$ and $(S_2,\preorderle_2)$ be preordered semirings and let $\varphi:S_1\to S_2$ be a monotone homomorphism. We define the map $\Delta(\varphi):\Delta(S_2,\preorderle_2)\to\Delta(S_1,\preorderle_1)$ as $f\mapsto f\circ\varphi$.
\end{definition}

\begin{lemma}\label{lem:spectrumproperties}
Let $(S,\preorderle)$ be a preordered semiring.
\begin{enumerate}[(i)]
\item\label{it:closedembedding} $\ev_S$ is a closed embedding. In particular, $\Delta(S,\preorderle)$ is Tychonoff.
\item\label{it:locallycompact} Suppose that $S$ is generated by a single element $u$. Then $\ev_u$ is proper and $\Delta(S,\preorderle)$ is locally compact.
\end{enumerate}
Let $(S_1,\preorderle_1)$ and $(S_2,\preorderle_2)$ be preordered semirings and $\varphi:S_1\to S_2$ a monotone semiring homomorphism.
\begin{enumerate}[resume*]
\item\label{it:continuous} $\Delta(\varphi)$ is continuous.
\item\label{it:proper} If for every $x_2\in S_2$ there is an $x_1\in S_1$ such that $x_2\preorderle_2\varphi(x_1)$ then $\Delta(\varphi)$ is proper.
\end{enumerate}
\end{lemma}
\begin{proof}
\ref{it:closedembedding}: The evaluation map $\ev_S$ clearly separates the points of $\Delta(S,\preorderle)$, thus gives an embedding into the Tychonoff space $\mathbb{R}_{\ge 0}^S$. The property of being a monotone semiring homomorphism is preserved by pointwise limits, therefore the embedding is closed.
\ref{it:locallycompact}: $S$ consists of polynomials in $u$ with coefficients from $\mathbb{N}$ and $\ev_{p(u)}=p(\ev_u)$. This means that $\ev_u$ already separates the points and its continuity is equivalent to the continuity of all evaluation maps. Therefore $\ev_u:\Delta(S,\preorderle)\to\mathbb{R}_{\ge 0}$ is a closed embedding. This implies that $\ev_u$ is proper and $\Delta(S,\preorderle)$ is locally compact.
\ref{it:continuous}: $\setbuild{\ev_s^{-1}(U)}{s\in S_1,U\subseteq\mathbb{R}_{\ge0}\text{ open}}$ is a subbasis for the topology of $\Delta(S_1,\preorderle_1)$. The preimage of such a set under $\Delta(\varphi)$ is $\ev_{\varphi(s)}^{-1}(U)$, which is open.
\ref{it:proper}: Let $C_1\subseteq\Delta(S_1,\preorderle_1)$ be compact and let $C_2=\Delta(\varphi)^{-1}(C_1)\subseteq\Delta(S_2,\preorderle_2)$. $C_1$ is a compact subset of a Hausdorff space, therefore closed, which implies that $C_2$ is also closed. For all $x_2\in S_2$ let
\begin{equation}
B_{x_2}=\inf_{\substack{x_1\in S_1  \\  x_2\preorderle_2\varphi(x_1)}}\sup_{f\in C_1}f(x_1).
\end{equation}
The infimum is by assumption over a nonempty set and the supremum is of a continuous function ($\ev_{x_1}$) over a compact set, therefore $B_{x_2}$ is finite. If $f\in C_2$ then $f\circ\varphi\in C_1$ so $f(x_2)\in[0,B_{x_2}]$. Thus $\ev_{S_2}$ embeds $C_2$ as a closed subset of the compact space $\prod_{x_2\in S_2}[0,B_{x_2}]$.
\end{proof}

\begin{proposition}\label{prop:polygrowthlocallycompact}
Let $(S,\preorderle)$ be a preordered semiring of polynomial growth and suppose that $u\in S$ is power universal. Then $\Delta(S,\preorderle)$ is locally compact and $\ev_u:\Delta(S,\preorderle)\to\mathbb{R}_{\ge0}$ is proper.
\end{proposition}
\begin{proof}
Let $S_0$ be the subsemiring generated by $u$ and $i:S_0\to S$ the inclusion. By \ref{it:locallycompact} and \ref{it:proper} of Lemma~\ref{lem:spectrumproperties}, $\Delta(i):\Delta(S,\preorderle)\to\Delta(S_0,\preorderle)$ is a proper map into a locally compact space $\Delta(S_0,\preorderle)$, thus $\Delta(S,\preorderle)$ is also locally compact.

$\ev^{S}_u=\ev^{S_0}_u\circ\Delta(i)$ (note that $u$ may be regarded as an element of both semirings and the domain of the evaluation map is the spectrum of the semiring indicated in the superscript) is a composition of proper maps, therefore also proper.
\end{proof}

\begin{lemma}\label{lem:spectrumofasymptoticpreorder}
Let $(S,\preorderle)$ be a preordered semiring of polynomial growth. Let $j:(S,\preorderle)\to(S,\asymptoticle)$ be the monotone homomorphism with underlying homomorphism the identity. Then $\Delta(j):\Delta(S,\asymptoticle)\to\Delta(S,\preorderle)$ is a homeomorphism.
\end{lemma}
\begin{proof}
$\Delta(j)$ is injective since for $f\in\Delta(S,\asymptoticle)$ we have $\Delta(j)(f)(x)=f(j(x))=f(x)$ for all $x\in S$. To see that $\Delta(j)$ is surjective we show that every $f\in\Delta(S,\preorderle)$ is also monotone under $\asymptoticle$. Suppose that $x\asymptoticge y$. Let $(k_n)_{n\in\mathbb{N}}$ be as in Definition~\ref{def:asymptoticge}, i.e. $k_n/n\to 0$ and for all $n\in\mathbb{N}$ $u^{k_n}x^n\ge y^n$. Then for every $f\in\Delta(S,\preorderle)$ we have
\begin{equation}
f(u)^{k_n}f(x)^n\ge f(y)^n.
\end{equation}
After taking roots and letting $n$ go to infinity we get $f(x)\ge f(y)$.

Thus $\Delta(j)$ is a bijection and can be used to identify the two spectra as sets. Under this identification the evaluation maps are the same in both cases, and so are the topologies they generate. Therefore $\Delta(j)$ is a homeomorphism.
\end{proof}

\begin{lemma}\label{lem:spectrumlocalization}
Let $(S,\preorderle)$ be a preordered semiring of polynomial growth and $T\subseteq S\setminus\{0\}$ a multiplicative set containing $1$. Let $j:S\mapsto T^{-1}S$ be the canonical map $x\mapsto\frac{x}{1}$. Then $\Delta(j):\Delta(T^{-1}S,\preorderle)\to\Delta(S,\preorderle)$ is a homeomorphism.
\end{lemma}
\begin{proof}
We prove that $\Delta(j)$ is injective. Let $\tilde{f}_1,\tilde{f}_2\in\Delta(T^{-1}S,\preorderle)$ be different. Then there are $s\in S$ and $t\in T$ such that $\tilde{f}_1(\frac{s}{t})\neq \tilde{f}_2(\frac{s}{t})$. Since $\frac{1}{1}=\frac{t}{1}\frac{1}{t}$ we have $\tilde{f}_1\left(\frac{t}{1}\right)\neq 0$ and $\tilde{f}_2\left(\frac{t}{1}\right)\neq 0$. If $\tilde{f}_1\left(\frac{t}{1}\right)\neq\tilde{f}_2\left(\frac{t}{1}\right)$ then $\Delta(j)(\tilde{f}_1)$ and $\Delta(j)(\tilde{f}_2)$ are different at $t$. Otherwise we have
\begin{equation}
\tilde{f}_1\left(\frac{s}{1}\right)=\tilde{f}_1\left(\frac{s}{t}\right)\tilde{f}_1\left(\frac{t}{1}\right)\neq\tilde{f}_2\left(\frac{s}{t}\right)\tilde{f}_2\left(\frac{t}{1}\right)=\tilde{f}_2\left(\frac{s}{1}\right),
\end{equation}
and therefore $\Delta(j)(\tilde{f}_1)$ and $\Delta(j)(\tilde{f}_2)$ are different at $s$.

We prove that $\Delta(j)$ is surjective. Let $f\in\Delta(S,\preorderle)$. Let $u\in S$ be power universal. For $s\in S\setminus\{0\}$ there is a $k\in\mathbb{N}$ such that $1\preorderle u^ks$. Applying $f$ to both sides we get $1\le f(u)^kf(s)$, which implies $f(s)>0$. If $\frac{s_1}{t_1}=\frac{s_2}{t_2}$ then by definition $rs_1t_2=rs_2t_1$ for some $r\in T$, therefore $f(r)f(s_1)f(t_2)=f(r)f(s_2)f(t_1)$, i.e. $f(s_1)f(t_2)=f(s_2)f(t_1)$. This means that the equality
\begin{equation}
\tilde{f}\left(\frac{s}{t}\right)=\frac{f(s)}{f(t)}
\end{equation}
gives a well-defined function on $T^{-1}S$. We claim that $\tilde{f}\in\Delta(T^{-1}S,\preorderle)$ and $\Delta(j)(\tilde{f})=f$. For $\frac{s_1}{t_1},\frac{s_2}{t_2}\in T^{-1}S$ we have
\begin{equation}
\begin{split}
\tilde{f}\left(\frac{s_1}{t_1}+\frac{s_2}{t_2}\right)
 & = \tilde{f}\left(\frac{s_1t_2+s_2t_1}{t_1t_2}\right)  \\
 & = \frac{f(s_1t_2+s_2t_1)}{f(t_1t_2)}=\frac{f(s_1)}{f(t_1)}+\frac{f(s_2)}{f(t_2)}=\tilde{f}\left(\frac{s_1}{t_1}\right)+\tilde{f}\left(\frac{s_1}{t_1}\right)
\end{split}
\end{equation}
and
\begin{equation}
\tilde{f}\left(\frac{s_1}{t_1}\cdot\frac{s_2}{t_2}\right)=\tilde{f}\left(\frac{s_1s_2}{t_1t_2}\right)=\frac{f(s_1s_2)}{f(t_1t_2)}=\frac{f(s_1)}{f(t_1)}\frac{f(s_2)}{f(t_2)}=\tilde{f}\left(\frac{s_1}{t_1}\right)\tilde{f}\left(\frac{s_1}{t_1}\right).
\end{equation}
Let $\frac{s_1}{t_1}\preorderge\frac{s_2}{t_2}$. This means there exists $r\in T$ such that $rs_1t_2\preorderge rs_2t_1$. Since $f$ is monotone and $f(r)>0$, this implies
\begin{equation}
\tilde{f}\left(\frac{s_1}{t_1}\right)=\frac{f(s_1)}{f(t_1)}\ge\frac{f(s_2)}{f(t_2)}=\tilde{f}\left(\frac{s_2}{t_2}\right).
\end{equation}
If $x\in S$ then clearly
Clearly $\tilde{f}(j(x))=\tilde{f}(\frac{x}{1})=f(x)$, thus $\tilde{f}$ extends $f$. In particular, $\tilde{f}(\frac{0}{1})=0$ and $\tilde{f}(\frac{1}{1})=1$.

Thus we may identify the two spectra via $\Delta(j)$ as sets. Since for $t\in T$ the map $\ev_t$ vanishes nowhere, the topology generated by the maps $(\ev_s)_{s\in S}$ and $(\frac{\ev_s}{\ev_t})_{s\in S,t\in T}$ is the same, therefore $\Delta(j)$ is a homeomorphism.
\end{proof}

Our goal in the following is to relate the spectral points of a semiring to those of a subsemiring. We will make use of a relaxed preorder whose restriction to the subsemiring is total.
\begin{definition}
Let $(S,\preorderle)$ be a preordered semiring, $S_0\le S$ a subsemiring and $f\in\Delta(S_0,\preorderle)$. We define the relation $\preorderle_f:=\preorderle_{R_f}$ (see Definition~\ref{def:Rpreorderdef}) where $R_f$ is the relation
\begin{equation}
R_f=\setbuild{(x,y)\in S_0\times S_0}{f(x)\le f(y)}\subseteq S\times S.
\end{equation}
\end{definition}
Since $\preorderle\subseteq\preorderle_f$ by Lemma~\ref{lem:Rpreorderproperties}, we may identify $\Delta(S,\preorderle_f)$ with a subset of $\Delta(S,\preorderle)$.
\begin{lemma}\label{lem:inclusionpreimage}
Let $(S,\preorderle)$ be a preordered semiring, $S_0\le S$ a subsemiring and $f\in\Delta(S_0,\preorderle)$. Let $i:S_0\hookrightarrow S$ be the inclusion. Then $\Delta(S,\preorderle_f)=\Delta(i)^{-1}(f)$.
\end{lemma}
\begin{proof}
Let $\tilde{f}\in\Delta(S,\preorderle_f)$ and $x\in S_0$. For all $n\in\mathbb{N}$ we have $\lfloor f(nx)\rfloor\le f(nx)\le\lceil f(nx)\rceil$, therefore by \ref{it:RpreordercontainsR} of Lemma~\ref{lem:Rpreorderproperties} also
\begin{equation}
\lfloor f(nx)\rfloor\preorderle_f nx\preorderle_f\lceil f(nx)\rceil.
\end{equation}
Apply $\tilde{f}$ to both sides, divide by $n$ and let $n\to\infty$ to get
\begin{equation}
f(x)=\lim_{n\to\infty}\frac{\lfloor f(nx)\rfloor}{n}\le\tilde{f}(x)\le\lim_{n\to\infty}\frac{\lceil f(nx)\rceil}{n}=f(x).
\end{equation}
Therefore $\tilde{f}$ agrees with $f$ on $S_0$, which means $\tilde{f}\in\Delta(i)^{-1}(f)$.

Let $\tilde{f}\in\Delta(i)^{-1}(f)$. Then $\tilde{f}$ is a semiring homomorphism and we need to show that it is monotone with respect to $\preorderle_f$. Let $a\preorderle_f b$. This means there are $x_1,\ldots,x_n,y_1,\ldots,y_n\in S_0$ and $s_1,\ldots,s_n\in S$ for some $n\in\mathbb{N}$ such that $\forall j\in[n]:f(x_j)\le f(y_j)$ and
\begin{equation}
a+s_1y_1+\cdots+s_ny_n\preorderle b+s_1x_1+\cdots+s_nx_n.
\end{equation}
Apply $\tilde{f}$ to both sides and rearrange as
\begin{equation}
\tilde{f}(s_1)\left(\tilde{f}(y_1)-\tilde{f}(x_1)\right)+\cdots+\tilde{f}(s_n)\left(\tilde{f}(y_n)-\tilde{f}(x_n)\right)\le \tilde{f}(b)-\tilde{f}(a).
\end{equation}
Since
\begin{equation}
\tilde{f}(y_j)-\tilde{f}(x_j)=f(y_j)-f(x_j)\ge 0
\end{equation}
for all $j$, this implies $\tilde{f}(b)\ge\tilde{f}(a)$, i.e. $\tilde{f}$ is monotone with respect to $\preorderle_f$.
\end{proof}

\begin{lemma}\label{lem:Rpreorderspectrum}
Let $(S,\preorderle)$ be a preordered semiring and $S_0\le S$ a subsemiring such that $\forall s\in S\setminus\{0\}\exists r,q\in S_0$ such that $1\preorderle rs\preorderle q$. Let $R\subseteq S_0\times S_0\subseteq S\times S$ be an arbitrary relation. Then
\begin{equation}
\Delta(S_0,\left.\preorderle_R\right|_{S_0})=\setbuild{f\in\Delta(S_0,\left.\preorderle\right|_{S_0})}{\forall(x,y)\in R:f(x)\le f(y)}.
\end{equation}
\end{lemma}
We emphasize that $\left.\preorderle_R\right|_{S_0}$ is in general not the same preorder as $(\left.\preorderle\right|_{S_0})_R$, since the latter would only allow elements $s_1,\ldots,s_n\in S_0$ in Definition~\ref{def:Rpreorderdef}.
\begin{proof}
First note that the condition $\forall s\in S\setminus\{0\}\exists r,q\in S_0$ such that $1\preorderle rs\preorderle q$ implies that every $f\in\Delta(S_0,\left.\preorderle\right|_{S_0})$ and $s\in S_0\setminus\{0\}$ satisfies $f(s)\neq 0$. To see this choose $r\in S_0$ such that $1\preorderle rs$ and apply $f$ to both sides.

We prove $\Delta(S_0,\left.\preorderle_R\right|_{S_0})\subseteq\setbuild{f\in\Delta(S_0,\left.\preorderle\right|_{S_0})}{\forall(x,y)\in R:f(x)\le f(y)}$. The preorders satisfy $\preorderle\subseteq\preorderle_R$, therefore $\Delta(S_0,\left.\preorderle_R\right|_{S_0})\subseteq\Delta(S_0,\left.\preorderle\right|_{S_0})$. Let $f\in\Delta(S_0,\left.\preorderle_R\right|_{S_0})$ and $(x,y)\in R$. \ref{it:RpreordercontainsR} of Lemma~\ref{lem:Rpreorderproperties} implies that $x\preorderle_R y$, therefore $f(x)\le f(y)$.

We prove $\Delta(S_0,\left.\preorderle_R\right|_{S_0})\supseteq\setbuild{f\in\Delta(S_0,\left.\preorderle\right|_{S_0})}{\forall(x,y)\in R:f(x)\le f(y)}$. By \ref{it:Rpreorderfinite} of Lemma~\ref{lem:Rpreorderproperties} it is enough to prove the statement for $|R|<\infty$. We prove by induction on $|R|$. If $R=\emptyset$ then $\preorderle_R=\preorderle$ and there is nothing to prove. Otherwise choose $(x,y)\in R$ and let $R_1=R\setminus\{(x,y)\}$ and $R_2=\{(x,y)\}$. Suppose that
\begin{equation}
\begin{split}
f
 & \in\setbuild{f'\in\Delta(S_0,\left.\preorderle\right|_{S_0})}{\forall(x',y')\in R:f'(x')\le f'(y')}  \\
 & =\setbuild{f'\in\Delta(S_0,\left.\preorderle_{R_1}\right|_{S_0})}{f'(x)\le f'(y)}
\end{split}
\end{equation}
Let $a,b\in S_0$ such that $a\preorderle_Rb$. This means
\begin{equation}\label{eq:basecase}
a+sy\preorderle_{R_1}b+sx
\end{equation}
for some $s\in S$ (\ref{it:Rpreorderunion} of Lemma~\ref{lem:Rpreorderproperties}). If $sx=sy=0$ then apply $f$ to both sides to get $f(a)\le f(b)$. Otherwise $s\neq 0$ and at least one of $x,y$ is nonzero. We prove
\begin{equation}\label{eq:toprove}
a\left(\sum_{m=0}^n x^my^{n-m}\right)+sy^{n+1}\preorderle_{R_1} b\left(\sum_{m=0}^n x^my^{n-m}\right)+sx^{n+1}
\end{equation}
by induction on $n$. The $n=0$ base case is \eqref{eq:basecase}. Suppose \eqref{eq:toprove} holds with $n-1$ instead of $n$. Then
\begin{equation}
\begin{split}
a\left(\sum_{m=0}^n x^my^{n-m}\right)+sy^{n+1}
 & = ax^n+y\left[a\left(\sum_{m=0}^{n-1} x^my^{n-1-m}\right)+sy^{n}\right]  \\
 & \preorderle_{R_1} ax^n+y\left[b\left(\sum_{m=0}^{n-1} x^my^{n-1-m}\right)+sx^{n}\right]  \\
 & = x^n(a+sy)+b\left(\sum_{m=0}^{n-1} x^my^{n-m}\right)  \\
 & \preorderle_{R_1} x^n(b+sx)+b\left(\sum_{m=0}^{n-1} x^my^{n-m}\right)  \\
 & = b\left(\sum_{m=0}^n x^my^{n-m}\right)+sx^{n+1},
\end{split}
\end{equation}
where the first inequality uses the induction hypothesis and the second one uses \eqref{eq:basecase}.

Let $r,q\in S_0$ such that $1\preorderle rs\preorderle q$. Then
\begin{equation}
ra\left(\sum_{m=0}^n x^my^{n-m}\right)+y^{n+1}\preorderle rb\left(\sum_{m=0}^n x^my^{n-m}\right)+qx^{n+1}.
\end{equation}
Apply $f$ to both sides and rearrange to get
\begin{equation}
f(r)\left(\sum_{m=0}^n f(x)^m f(y)^{n-m}\right)\left(f(a)-f(b)\right)\le f(q)f(x)^{n+1}-f(y)^{n+1}
\end{equation}

Divide by the coefficient of $f(a)-f(b)$ (nonzero since $x$ or $y$ is nonzero and $r\neq 0$). Then we use $f(x)\le f(y)$ and $f(q)\ge 1$ to get
\begin{equation}
\begin{split}
f(a)-f(b)
 & \le \frac{f(q)f(x)^{n+1}-f(y)^{n+1}}{f(r)\left(\sum_{m=0}^n f(x)^m f(y)^{n-m}\right)}  \\
 & \le \frac{(f(q)-1)f(x)^{n+1}}{f(r)\left(\sum_{m=0}^n f(x)^m f(y)^{n-m}\right)}  \\
 & \le \frac{(f(q)-1)f(x)^{n+1}}{f(r)(n+1)f(x)^n}  \\
 & = \frac{(f(q)-1)f(x)}{f(r)}\frac{1}{n+1}.
\end{split}
\end{equation}
This inequality is true for every $n\in\mathbb{N}$, therefore $f(a)\le f(b)$.
\end{proof}

\begin{proposition}\label{prop:surjective}
Let $(S,\preorderle)$ be a preordered semiring of polynomial growth and $S_0\le S$ a subsemiring satisfying $\forall s\in S\setminus\{0\}\exists r,q\in S_0$ such that $1\preorderle rs\preorderle q$. Let $i:S_0\hookrightarrow S$ be the inclusion. Then $\Delta(i)$ is surjective.
\end{proposition}
\begin{proof}
We can assume that the inclusion $\mathbb{N}\hookrightarrow S$ is an order embedding (otherwise both spectra are empty).

By Lemma~\ref{lem:inclusionpreimage} it is enough to show that for every $f\in\Delta(S_0,\preorderle)$ the set $\Delta(S,\preorderle_f)$ is nonempty. Let $u\in S$ be power universal. By the assumption on $S_0$ there is an $u'\in S_0$ such that $u'\preorderge u$, and any such $u'$ is also power universal in $S$. By \ref{it:RpreordercontainsR} of Lemma~\ref{lem:Rpreorderproperties},
\begin{equation}
u'\preorderle_f 2^{\lceil\log_2 f(u')\rceil},
\end{equation}
therefore $2$ is also power universal for $(S,\preorderle_f)$.

Lemma~\ref{lem:Rpreorderspectrum} implies that $\Delta(S_0,\left.\preorderle_f\right|_{S_0})=\{f\}$. In particular, for $n,m\in\mathbb{N}$ we have $n\preorderle_f m$ iff $n\le m$. Therefore we can apply Theorem~\ref{thm:Strassen} to the preordered semiring $(S,\preorderle_f)$ and conclude $\Delta(S,\preorderle_f)\neq\emptyset$.
\end{proof}

\section{Proof of main result and uniqueness}\label{sec:mainproof}

Now we have all the technical tools to prove Theorem~\ref{thm:main}. In the setting of that theorem, we introduce the following notations:
\begin{align}
S_+ & = \setbuild{s\in S}{\exists n\in\mathbb{N}:ns\preorderge 1}\cup\{0\}  \\
S_- & = \setbuild{s\in S}{\exists n\in\mathbb{N}:n\preorderge s}  \\
S_b & = S_+\cap S_-.
\end{align}
We will see that all three are subsemirings, and the definition of $S_b$ ensures that it satisfies the conditions of Theorem~\ref{thm:Strassen}. Given a pair of elements in $S$, we can multiply both with the same element of $M$ to get elements of $S_-$, at least one of them in $S_b$. Here Strassen's theorem ensures that the collection of monotone homomorphisms characterize the asymptotic preorder. An essential part of the proof is to show that most spectral points of $S_b$ admit an extension to $S$. Informally, an extension exists unless it would evaluate to $\infty$ on the power universal element. More precisely, by assumption~\ref{it:invertibleuptobounded} we can form an ``inverse up to bounded elements'' $\bar{u}$ of $u$, and $f\in\Delta(S_b,\preorderle)$ extends to $S$ iff $f(1+\bar{u})>1$.

\begin{lemma}\label{lem:plusminusproperties}
\leavevmode
\begin{enumerate}[(i)]
\item\label{it:plusminussubsemiring} $S_+$, $S_-$ and $S_b$ are subsemirings of $S$.
\item\label{it:boundedtominusextension} Let $i:S_b\to S_-$ denote the inclusion. Then $\Delta(i):\Delta(S_-,\preorderle)\to\Delta(S_b,\preorderle)$ is a homeomorphism.
\end{enumerate}
\end{lemma}
\begin{proof}
\ref{it:plusminussubsemiring}: $0$ and $1$ are clearly contained both in $S_+$ and $S_-$. Let $s_1,s_2\in S_+$. This means that there are $n_1,n_2\in\mathbb{N}$ such that $n_1s_1\preorderge 1$ and $n_2s_2\preorderge 1$. Then $n_1(s_1+s_2)\preorderge n_1s_1\preorderge 1$ and $n_1n_2(s_1s_2)\preorderge 1\cdot 1=1$, therefore $s_1+s_2\in S_+$ and $s_1s_2\in S_+$. Let $s_1,s_2\in S_-$. This means that there are $n_1,n_2\in\mathbb{N}$ such that $n_1\preorderge s_1$ and $n_2\preorderge s_2$. Then $n_1+n_2\preorderge s_1+s_2$ and $n_1n_1\preorderge s_1s_2$, therefore $s_1+s_2\in S_-$ and $s_1s_2\in S_-$. $S_b=S_+\cap S_-$ is the intersection of subsemirings, therefore it is also a subsemiring.

\ref{it:boundedtominusextension}: We prove that $\Delta(i)$ is injective. Let $\tilde{f}\in\Delta(S_-,\preorderle)$. If $x\in S_-$ then $1+x\in S_b$ and therefore $\tilde{f}(x)=\tilde{f}(1+x)-1=\Delta(i)(\tilde{f})(1+x)-1$, so $\tilde{f}$ can be reconstructed from its restriction $\Delta(i)(\tilde{f})$.

We prove that $\Delta(i)$ is surjective. Let $f\in\Delta(S_b,\preorderle)$ and let $\tilde{f}:S_-\to\mathbb{R}$ be defined as $\tilde{f}(x)=f(1+x)-1$. Since $1\preorderle 1+x$, we have $\tilde{f}(x)\ge f(1)-1=0$. We show that $\tilde{f}$ is a monotone semiring homomorphism and $\Delta(i)(\tilde{f})=f$. Clearly $\tilde{f}(0)=f(1+0)-1=1-1=1$ and $\tilde{f}(1)=f(1+1)-1=2-1=1$.

We prove additivity.
\begin{equation}
\begin{split}
\tilde{f}(x+y)
 & = f(1+x+y)-1  \\
 & = f(1+x+1+y)-1-1  \\
 & = f(1+x)-1+f(1+y)-1  \\
 & = \tilde{f}(x)+\tilde{f}(y)
\end{split}
\end{equation}

We prove multiplicativity.
\begin{equation}
\begin{split}
\tilde{f}(xy)
 & = f(1+xy)-1  \\
 & = f(1+x+1+y+1+xy)-f(1+x)-f(1+y)-1  \\
 & = f(1+x+y+xy)-f(1+x)-f(1+y)+1  \\
 & = f((1+x)(1+y))-f(1+x)-f(1+y)+1  \\
 & = (f(1+x)-1)(f(1+y)-1)  \\
 & = \tilde{f}(x)\tilde{f}(y)
\end{split}
\end{equation}

We prove that $\tilde{f}$ is monotone. Let $x,y\in S_-$ and suppose that $x\preorderle y$. Then $1+x\preorderle 1+y$, therefore
\begin{equation}
\tilde{f}(x)=f(1+x)\le f(1+y)=\tilde{f}(y).
\end{equation}

If $x\in S_b$ then $\tilde{f}(x)=f(1+x)-1=f(1)+f(x)-1=f(x)$, so $\Delta(i)(\tilde{f})=f$.

Finally, from the equality $\tilde{f}(x)=\tilde{f}(1+x)-1$ we see that pointwise convergence in $\Delta(S_-,\preorderle)$ is equivalent to pointwise convergence of the restrictions to $S_b$.
\end{proof}

\begin{lemma}\label{lem:extensionfromminus}
Let $u$ be power universal in $S$ and suppose that there is a $\bar{u}\in S\setminus\{0\}$ such that $u\bar{u}\in S_b$. Then $\bar{u}\in S_-$ and for any $f\in\Delta(S_-,\preorderle)$ the following are equivalent:
\begin{enumerate}[(i)]
\item\label{it:extends} $f$ has an extension $\tilde{f}:S\to\mathbb{R}_{\ge 0}$ such that $\tilde{f}\in\Delta(S,\preorderle)$
\item\label{it:onminusnonzero} $\forall x\in S_-\setminus\{0\}:f(x)\neq 0$.
\item\label{it:ubarnonzero} $f(\bar{u})\neq 0$
\end{enumerate}
When an extension exists, it is unique.
\end{lemma}
\begin{proof}
Since $1\preorderle u$ and $u\bar{u}\in S_b$, there is an $n\in\mathbb{N}$ such that $\bar{u}\preorderle u\bar{u}\preorderle n$, therefore $\bar{u}\in S_-$.

\ref{it:extends}$\implies$\ref{it:onminusnonzero}: Let $x\in S_-\setminus\{0\}$ and choose $k\in\mathbb{N}$ such that $1\preorderle u^kx$. Then $1\le\tilde{f}(u^kx)=\tilde{f}(u)^kf(x)$, therefore $f(x)\neq 0$.

\ref{it:onminusnonzero}$\implies$\ref{it:ubarnonzero}: We have seen that $\bar{u}\in S_-$. It is also nonzero since there is an $n$ such that $1\preorderle nu\bar{u}$.

\ref{it:ubarnonzero}$\implies$\ref{it:extends}: Let $x\in S$. There is a $k\in\mathbb{N}$ such that $x\preorderle u^k$, and therefore $\bar{u}^k x\preorderle (u\bar{u})^k$, which implies $\bar{u}^k x\in S_-$. From this we see that if an extension $\tilde{f}$ exists, it must satisfy $\tilde{f}(x)=f(\bar{u})^{-k}f(\bar{u}^k x)$, which proves uniqueness. We prove that this expression is well defined. If $\bar{u}^{k_1} x\in S_-$ and $\bar{u}^{k_2} x\in S_-$ with $k_1<k_2$ then
\begin{equation}
\begin{split}
f(\bar{u})^{-k_2}f(\bar{u}^{k_2} x)
 & = f(\bar{u})^{-k_2}f(\bar{u}^{k_2-k_1}\bar{u}^{k_1} x)  \\
 & = f(\bar{u})^{-k_2}f(\bar{u}^{k_2-k_1})f(\bar{u}^{k_1} x)=f(\bar{u})^{-k_1}f(\bar{u}^{k_1} x).
\end{split}
\end{equation}
$\tilde{f}$ extends $f$ since for $x\in S_-$ one can take $k=0$ above.

We show that $\tilde{f}$ is monotone. Let $x_1,x_2\in S$, $x_1\preorderle x_2$. Choose $k\in\mathbb{N}$ such that $\bar{u}^kx_1\in S_-$ and $\bar{u}^kx_2\in S_-$. Then $\bar{u}^kx_1\preorderle \bar{u}^kx_2$, therefore
\begin{equation}
\tilde{f}(x_1)=f(\bar{u})^{-k}f(\bar{u}^k x_1)\le f(\bar{u})^{-k}f(\bar{u}^k x_2)=\tilde{f}(x_2).
\end{equation}

We show that $\tilde{f}$ is multiplicative. Let $x_1,x_2\in S$ and choose $k_1,k_2\in\mathbb{N}$ such that $\bar{u}^{k_1}x_1\in S_-$ and $\bar{u}^{k_2}x_2\in S_-$. Then $\bar{u}^{k_1+k_2}x_1x_2\in S_-$ and
\begin{equation}
\begin{split}
\tilde{f}(x_1x_2)
 & = f(\bar{u})^{-{k_1+k_2}}f(\bar{u}^{k_1+k_2}x_1x_2)  \\
 & = f(\bar{u})^{-k_1}f(\bar{u}^{k_1} x)f(\bar{u})^{-k_2}f(\bar{u}^{k_2} x)=\tilde{f}(x_1)\tilde{f}(x_2).
\end{split}
\end{equation}

We show that $\tilde{f}$ is additive. Let $x_1,x_2\in S$ and choose $k_1,k_2\in\mathbb{N}$ such that $\bar{u}^{k_1}x_1\in S_-$ and $\bar{u}^{k_2}x_2\in S_-$. Then $\bar{u}^{k_1+k_2}(x_1+x_2)\in S_-$ and
\begin{equation}
\begin{split}
\tilde{f}(x_1+x_2)
 & =f(\bar{u})^{-{k_1+k_2}}f(\bar{u}^{k_1+k_2}(x_1+x_2))  \\
 & =f(\bar{u})^{-{k_1+k_2}}\left(f(\bar{u}^{k_1+k_2}x_1)+f(\bar{u}^{k_1+k_2}x_2)\right)=\tilde{f}(x_1)+\tilde{f}(x_2).
\end{split}
\end{equation}
\end{proof}
According to Lemma~\ref{lem:plusminusproperties} and Lemma~\ref{lem:extensionfromminus} we may make the identification $\Delta(S,\preorderle)\subseteq\Delta(S_-,\preorderle)=\Delta(S_b,\preorderle)$.

\begin{proof}[Proof of Theorem~\ref{thm:main}]
The implication $x\asymptoticge y\implies\forall f\in\Delta(S,\preorderle):f(x)\ge f(y)$ follows from Lemma~\ref{lem:spectrumofasymptoticpreorder}.

For the other direction, suppose that $x,y\in S$ satisfy $\forall f\in\Delta(S,\preorderle):f(x)\ge f(y)$. Let $m_1\in M$ be such that $m_1y\in S_b\setminus\{0\}$ and let $m_2\in M$ be such that $m_2m_1x\in S_b\setminus\{0\}$. Then for all $f\in\Delta(S,\preorderle)$ we have
\begin{equation}
f(m_2)=\frac{f(m_2m_1x)}{f(m_1)f(x)}\le\frac{f(m_2m_1x)}{f(m_1)f(y)}=\frac{f(m_2m_1x)}{f(m_1y)}.
\end{equation}
The numerator of the right hand side is bounded from above and the denominator is bounded away from $0$, therefore $\ev_{m_2}$ is bounded on $\Delta(S,\preorderle)$. By Proposition~\ref{prop:surjective} it is also bounded on $\Delta(S_0,\preorderle)$, and thus by assumption~\ref{it:boundedev} we have $m_2\in S_-$.

Let $t=m_1m_2$ so that $tx\in S_b$ and $ty\in S_-$ (see \ref{it:plusminussubsemiring} of Lemma~\ref{lem:plusminusproperties}). Let $u$ be power universal and $\bar{u}\in M$ such that $u\bar{u}\in S_b\setminus\{0\}$ (possible by \ref{it:invertibleuptobounded}). Let $k\in\mathbb{N}$ such that $k\preorderge\bar{u}$ (Lemma~\ref{lem:extensionfromminus}) and let
\begin{equation}
\delta=\min_{f\in\Delta(S_b,\preorderle)}f(tx).
\end{equation}
For every $f\in\Delta(S,\preorderle)$ we have
\begin{equation}
\begin{split}
f((k+1)\lceil\delta^{-n}\rceil(tx)^n)
 & \ge f(\bar{u})\lceil\delta^{-n}\rceil f(tx)^n+1  \\
 & \ge f(\bar{u})\lceil\delta^{-n}\rceil f(ty)^n+1  \\
 & = f(\bar{u}\lceil\delta^{-n}\rceil(ty)^n+1),
\end{split}
\end{equation}
whereas for every $f\in\Delta(S_b,\preorderle)\setminus\Delta(S,\preorderle)$ we have $f(\bar{u})=0$ and therefore
\begin{equation}
\begin{split}
f((k+1)\lceil\delta^{-n}\rceil(tx)^n)\ge 1
 & = f(\bar{u})\lceil\delta^{-n}\rceil f(ty)^n+1  \\
 & = f(\bar{u}\lceil\delta^{-n}\rceil(ty)^n+1).
\end{split}
\end{equation}
We apply Theorem~\ref{thm:Strassen} to $S_b$ and infer
\begin{equation}
(k+1)\lceil\delta^{-n}\rceil(tx)^n\asymptoticge \bar{u}\lceil\delta^{-n}\rceil(ty)^n+1\preorderge\bar{u}\lceil\delta^{-n}\rceil(ty)^n,
\end{equation}
The factors $\lceil\delta^{-n}\rceil$ can be cancelled by \ref{it:asymptoticcancellation} of Lemma~\ref{lem:asymptoticproperties} and Lemma~\ref{lem:aasymptotic}. Thus we have $(k+1)(tx)^n\asymptoticge\bar{u}(ty)^n$ for all $n$, which implies $tx\asymptoticge ty$ by \ref{it:asymptoticsmallfactors} of Lemma~\ref{lem:asymptoticproperties}. Finally, we cancel the factors $t$ using \ref{it:asymptoticcancellation} of Lemma~\ref{lem:asymptoticproperties} once more and conclude $x\asymptoticge y$.

$\Delta(S,\preorderle)$ is locally compact and $\ev_u$ is proper on $\Delta(S,\preorderle)$ by Proposition~\ref{prop:polygrowthlocallycompact}.
\end{proof}

\begin{proof}[Proof of Corollary~\ref{cor:main}]
Consider the localization $T^{-1}S$ with its asymptotic preorder (see Lemma~\ref{lem:localizationasymptotic}). Choose $M'=\setbuild{\frac{mt_1}{t_2}}{m\in M,t_1,t_2\in T}$. Then the semiring generated by $M'$ is $\setbuild{\frac{s}{t}}{s\in S_0,t\in T}=T^{-1}S_0$. We use Theorem~\ref{thm:main} with the subset $M'\subseteq T^{-1}S$.

If $\frac{s}{t}\in T^{-1}S\setminus\{0\}$ (with $s\in S\setminus\{0\}$ and $t\in T$) then let $m\in M$, $t_1,t_2\in T$ and $n\in\mathbb{N}$ such that $t_2\preorderle n mt_1s$ and $mt_1s\preorderle n t_2$ as in the condition \ref{it:fracinvertibleuptobounded}. Then $1\preorderle n\frac{mt_1t}{t_2}\frac{s}{t}$ and $\frac{mt_1t}{t_2}\frac{s}{t}\preorderle n$, therefore \ref{it:invertibleuptobounded} is satisfied by $M'$.

Let $\frac{mt_1}{t_2}\in M'$ (with $m\in M$ and $t_1,t_2\in T$) such that $\ev_{\frac{mt_1}{t_2}}$ is bounded on $\Delta(T^{-1}S_0,\asymptoticle)$. Then $\ev_m\frac{\ev_{t_1}}{\ev_{t_2}}$ is bounded on $\Delta(S_0,\preorderle)$ (see Lemma~\ref{lem:spectrumofasymptoticpreorder} and Lemma~\ref{lem:spectrumlocalization}), so by \ref{it:fracboundedev} we have $mt_1\preorderle n t_2$ for some $n\in\mathbb{N}$. This implies $\frac{mt_1}{t_2}\asymptoticle n$, therefore \ref{it:boundedev} is satisfied.

Using Lemma~\ref{lem:localizationasymptotic} and $\Delta(S,\preorderle)=\Delta(S,\asymptoticle)=\Delta(T^{-1}S,\asymptoticle)$ (again by Lemma~\ref{lem:spectrumofasymptoticpreorder} and Lemma~\ref{lem:spectrumlocalization}), for $x,y\in S$ we conclude 
\begin{equation}
x\asymptoticge y\iff\frac{x}{1}\asymptoticge\frac{y}{1}\iff\forall f\in\Delta(S,\preorderle):f(x)\ge f(y).
\end{equation}
\end{proof}

\begin{proof}[Proof of Proposition~\ref{prop:uniqueness}]
We prove existence. Define the map $h(x):S\to\mathbb{R}_{\ge 0}$ by $h(x)(s)=\Phi(s)(x)$. Then $h(x)\in\Delta(S,\asymptoticle)=\Delta(S,\preorderle)$, $\Phi(s)=\ev_s\circ h$ and the map $h$ is injective (since $\Phi(S)$ separates points) and continuous (because $\forall s\in S:\Phi(s)$ is continuous). Let $u$ be power universal and consider the set
\begin{equation}
A=\setbuild{a\frac{\ev_s}{\ev_u^{k+1}}-b\frac{\ev_t}{\ev_u^{k+1}}}{a,b\in\mathbb{R}_{\ge 0},k\in\mathbb{N},\exists n\in\mathbb{N}\exists s,t\in S,s\preorderle nu^k,t\preorderle nu^k}
\end{equation}
Then $A\subseteq C_0(\Delta(S,\asymptoticle))$ is a subalgebra that separates points (if $f_1(u)\neq f_2(u)$ then $\frac{1}{\ev_u}\in A$ separates them, otherwise if $f_1(s)\neq f_2(s)$ then $\frac{\ev_s}{\ev_u^{k+1}}\in A$ for some $k$ does) and vanishes nowhere (e.g. $\frac{1}{\ev_u}\in A$ is nowhere zero), so by the Stone--Weierstrass theorem for locally compact spaces \cite[Theorem A.10.2]{deitmar2009principles} it is dense in $C_0(\Delta(S,\asymptoticle))$.

Suppose that $h$ is not surjective and let $f_0\in\Delta(S,\preorderle)\setminus h(X)$. Let $0<\epsilon<f_0(u)$. By continuity, $(\ev_u)^{-1}((f_0(u)-\epsilon,f_0(u)+\epsilon))$ is an open set containing $f_0$. Since $\Phi(u)$ is proper, the subset $h(\Phi(u)^{-1}([f_0(u)-\epsilon,f_0(u)+\epsilon]))\subseteq\Delta(S,\preorderle)$ is compact in a Hausdorff space and therefore closed, and does not contain $f_0$. The set
\begin{equation}
U=(\ev_u)^{-1}((f_0(u)-\epsilon,f_0(u)+\epsilon))\setminus h(\Phi(u)^{-1}([f_0(u)-\epsilon,f_0(u)+\epsilon])).
\end{equation}
is open, disjoint from $h(X)$, and contains $f_0$. By Urysohn's lemma for locally compact Hausdorff spaces \cite[Lemma A.8.1]{deitmar2009principles}, there is a function $g\in C_0(\Delta(S,\preorderle))$ that is $1$ at $f_0$ and $0$ outside $U$, i.e. vanishes on $h(X)$.

By the density of $A$, there are exist $N\in\mathbb{N}$, $s,t\in S$, $k\in\mathbb{N}$ such that $s\preorderle nu^k,t\preorderle nu^k$ for some $n\in\mathbb{N}$ and
\begin{align}
\frac{1}{2N}\left(\frac{f(s)}{f(u)^{k+1}}-\frac{f(t)}{f(u)^{k+1}}\right)
 & < \frac{1}{4}\qquad\text{for all $f\in h(X)$}  \\
\frac{1}{2N}\left(\frac{f_0(s)}{f_0(u)^{k+1}}-\frac{f_0(t)}{f_0(u)^{k+1}}\right)
 & > \frac{3}{4}.
\end{align}
so
\begin{align}
f(s)-f(t+Nu^{k+1})
 & <-\frac{N}{2}f(u)^{k+1}\qquad\text{for all $f\in h(X)$}  \\
f_0(s)-f_0(t+Nu^{k+1})
 & >\frac{N}{2}f_0(u)^{k+1}
\end{align}
This means that $\Phi(t+Nu^{k+1})\ge\Phi(s)$, so by assumption $t+Nu^{k+1}\asymptoticge s$. On the other hand, $f_0(t+Nu^{k+1})<f_0(s)$, a contradiction.

We prove uniqueness. Let $h_1,h_2:X\to\Delta(S,\preorderle)$ be two homeomorphisms satisfying $\forall s\in S:\Phi(s)=\ev_s\circ h_1=\ev_s\circ h_2$. This means that for every $x\in X$ and $s\in S$ the equality $h_1(x)(s)=h_2(x)(s)$ holds. Since $S$ separates the points of $\Delta(S,\preorderle)$, this implies $h_1(x)=h_2(x)$ for all $x\in X$, i.e. $h_1=h_2$.
\end{proof}

\section*{Acknowledgement}

I thank Tobias Fritz for providing useful feedback. This work was supported by the ÚNKP-19-4 New National Excellence Program of the Ministry for Innovation and Technology and the Bolyai J\'anos Research Fellowship of the Hungarian Academy of Sciences. I acknowledge support from the Hungarian National Research, Development and Innovation Office (NKFIH) within the Quantum Technology National Excellence Program (Project Nr.~2017-1.2.1-NKP-2017-00001) and via the research grants K124152, KH129601.

\bibliography{refs}{}

\end{document}